\documentclass[12pt, reqno]{amsart}
\UseRawInputEncoding
\usepackage{amsmath, amsthm, amscd, amsfonts, amssymb, graphicx, color}
\usepackage[bookmarksnumbered, colorlinks, plainpages]{hyperref}
\hypersetup{colorlinks=true,linkcolor=red, anchorcolor=green,
citecolor=cyan, urlcolor=red, filecolor=magenta, pdftoolbar=true}

\newtheorem{theorem}{Theorem}[section]
\newtheorem{lemma}[theorem]{Lemma}
\newtheorem{proposition}[theorem]{Proposition}
\newtheorem{corollary}[theorem]{Corollary}
\theoremstyle{definition}
\newtheorem{definition}[theorem]{Definition}

\theoremstyle{remark}
\newtheorem{remark}[theorem]{Remark}
\newtheorem{notation}[theorem]{Notation}
\numberwithin{equation}{section}
\usepackage{color}

\begin{document}

\setcounter{page}{1}

\title[The weighted Hilbert--Schmidt numerical radius]
{The weighted Hilbert--Schmidt numerical radius}

\author[A.~Zamani]
{Ali Zamani}

\address{School of Mathematics and Computer Sciences, Damghan University, Damghan, P.O.BOX 36715-364, Iran}
\email{zamani.ali85@yahoo.com}


\begin{abstract}
Let $\mathbb{B}(\mathcal{H})$ be the algebra of all bounded linear operators on a Hilbert space $\mathcal{H}$
and let $N(\cdot)$ be a norm on $\mathbb{B}(\mathcal{H})$.
For every $0\leq \nu \leq 1$, we introduce the $w_{_{(N,\nu)}}(A)$ as an extension of the classical numerical radius by
\begin{align*}
w_{_{(N,\nu)}}(A):= \displaystyle{\sup_{\theta \in \mathbb{R}}}
N\left(\nu e^{i\theta}A + (1-\nu)e^{-i\theta}A^*\right)
\end{align*}
and investigate basic properties of this notion and prove inequalities involving it.
In particular, when $N(\cdot)$ is the Hilbert--Schmidt norm ${\|\!\cdot\!\|}_{2}$,
we present several the weighted Hilbert--Schmidt numerical radius inequalities for operator matrices.
Furthermore, we give a refinement of the triangle inequality for the Hilbert--Schmidt norm as follows:
\begin{align*}
{\|A+B\|}_{2} \leq \sqrt{2w_{_{({\|\!\cdot\!\|}_{2},\nu)}}^2\left(\begin{bmatrix}
0 & A \\
B^* & 0
\end{bmatrix}\right) - (1-2\nu)^2{\|A-B\|}_{2}^2} \leq {\|A\|}_{2} + {\|B\|}_{2}.
\end{align*}
Our results extend some theorems due to F.~Kittaneh et al. (2019).
\end{abstract}
\maketitle
\section{Introduction and preliminaries}
Let $\mathbb{B}(\mathcal{H})$ denote the $C^{\ast}$-algebra of all bounded
linear operators on a complex Hilbert space $\big(\mathcal{H}, \langle \cdot, \cdot\rangle \big)$
and $I$ stand for the identity operator on $\mathcal{H}$.
Every operator $A\in \mathbb{B}(\mathcal{H})$ can be represented as $A = \mathfrak{R}(A) + i\mathfrak{I}(A)$,
the Cartesian decomposition, where $\mathfrak{R}(A)= \frac{A+A^*}{2}$ and $\mathfrak{I}(A)= \frac{A-A^*}{2i}$
are the real and imaginary parts of $A$, respectively.
For $0\leq \nu \leq 1$, we define the weighted real and imaginary parts of $A\in \mathbb{B}(\mathcal{H})$ by
$\mathfrak{R}_{\nu}(A)= \nu A + (1-\nu)A^*$ and $\mathfrak{I}_{\nu}(A)= \nu (-iA) + (1-\nu)(-iA)^*$, respectively.
When $\nu = 1/2$, we clearly have $\mathfrak{R}_{_{1/2}}(A)= \mathfrak{R}(A)$ and $\mathfrak{I}_{_{1/2}}(A)= \mathfrak{I}(A)$.
Let $\|A\|$ and ${\|A\|}_{2}$ denote the usual operator norm and the Hilbert--Schmidt norm of $A$, respectively.
Recall that $A$ belongs to the Hilbert--Schmidt class $\mathcal{C}_2(\mathcal{H})$
if ${\|A\|}_{2} = \sqrt{{\rm tr}(A^*A)}$ is finite, where the symbol ${\rm tr}$ denotes the usual trace.
It is known that for $A, B\in \mathcal{C}_2(\mathcal{H})$,
\begin{align}\label{I.12.1}
\left|{\rm tr}\left(AB\right)\right|\leq {\|A\|}_{2}{\|B\|}_{2}.
\end{align}
The numerical radius of $A\in\mathbb{B}(\mathcal{H})$ is defined by
\begin{align*}
w(A) = \sup \big\{|\langle Ax, x\rangle|:\,\,x\in\mathcal{H},\|x\| = 1\big\}.
\end{align*}
This concept is useful in studying linear operators and has attracted the attention of many authors
in the last few decades (e.g., see \cite{B.H.K, B.B.P, Bo.Co, G.R, Sa, S.B.B.P, Z.M.X.F}, and their references).

An important and useful identity for the numerical radius (see, e.g., \cite{Y}) is as follows:
\begin{align}\label{I.12.3}
w(A)= \displaystyle{\sup_{\theta \in \mathbb{R}}}
\left\|\mathfrak{R}(e^{i\theta}A)\right\|.
\end{align}

It is also well known that $w(\cdot)$ defines a norm on
$\mathbb{B}(\mathcal{H})$ such that for all $A\in\mathbb{B}(\mathcal{H})$,
\begin{align}\label{I.12.2}
\frac{1}{2}\|A\| \leq w(A)\leq \|A\|.
\end{align}
The inequalities in \eqref{I.12.2} are sharp.
The first inequality becomes an equality if $A$ is square-zero (i.e., $A^2 = 0$)
and the second inequality becomes an equality if $A$ is normal in the sense that $A^*A = AA^*$;
see \cite{Dra, Ga.Wu}.

Now, let $N(\cdot)$ be an arbitrary norm on $\mathbb{B}(\mathcal{H})$.
The norm $N(\cdot)$ is self-adjoint if $N(A^*)=N(A)$,
unitarily invariant if $N(VAU)=N(A)$ and
weakly unitarily invariant if $N(U^*AU)=N(A)$
for every $A\in\mathbb{B}(\mathcal{H})$ and unitary $U, V\in\mathbb{B}(\mathcal{H})$.
The norms $\|\!\cdot\!\|$ and ${\|\!\cdot\!\|}_{2}$ are typical examples of self-adjoint and unitarily invariant norms.
Further, the norm $w(\cdot)$ is self-adjoint and weakly unitarily invariant, but, it is not unitarily invariant.

Based on some operator theory studies on Hilbert spaces, several
generalizations for the concept of numerical radius have recently been introduced \cite{A.K.1, Ba.Ka.Ah, S.K.S, Z.W}.
One of these generalizations is the $\mathbb{A}$-numerical radius of an operator $A\in\mathbb{B}(\mathcal{H})$
defined by
\begin{align*}
w_{\mathbb{A}}(A) = \sup \big\{|\langle Ax, \mathbb{A}x\rangle|:\,\,x\in\mathcal{H}, \langle x, \mathbb{A}x\rangle=1\big\},
\end{align*}
see, e.g., \cite{Ba.Ka.Ah, Z.2}. Here, $\mathbb{A}$ is a positive bounded operator on $\mathcal{H}$.

In \cite{A.K.1}, in light of the equality \eqref{I.12.3}, Abu-Omar and Kittaneh introduced the so-called generalized numerical radius,
by introducing the quantity
\begin{align*}
w_{_{N}}(A)= \displaystyle{\sup_{\theta \in \mathbb{R}}}
N\left(\mathfrak{R}(e^{i\theta}A)\right),
\end{align*}
where $N(\cdot)$ is a given norm on $\mathbb{B}(\mathcal{H})$ and $A\in\mathbb{B}(\mathcal{H})$.

Another generalization of the classical numerical radius has been introduced
in \cite{Z.W}, where the authors defined the numerical radius $r_{_{N}}(A)$ induced as follows:
\begin{align*}
r_{_{N}}(A) = \sup\big\{|\xi|:\, \xi \in \mathbb{C}, \, I\perp^N_B (A - \xi I)\big\},
\end{align*}
in which the notation $\perp^N_B$ refers to Birkhoff--James orthogonality.

Very recently, Sheikhhosseini et al. \cite{S.K.S}
introduced the so-called the weighted numerical radius:
If $0\leq \nu \leq 1$, the weighted numerical radius for $A\in \mathbb{B}(\mathcal{H})$, denoted by $w_{_{\nu}}(A)$, is introduced by
\begin{align*}
w_{_{\nu}}(A)= \displaystyle{\sup_{\theta \in \mathbb{R}}}
\left\|\mathfrak{R}_{\nu}(e^{i\theta}A)\right\|.
\end{align*}
Hence, one can see that research involving possible generalizations of the
numerical radius with plausible properties has attracted researchers in this
field.

In Section 2 of this paper and inspired by \cite{A.K.1} and \cite{S.K.S} (see also \cite{AZa}),
for an arbitrary norm $N(\cdot)$ on $\mathbb{B}(\mathcal{H})$ and $0\leq \nu \leq 1$,
we define the $w_{_{(N,\nu)}}(\cdot)$ as a generalization of the weighted numerical radius
and investigate basic properties of this norm and prove inequalities involving it.
In Section 3, when $N(\cdot)$ is the Hilbert--Schmidt norm ${\|\!\cdot\!\|}_{2}$,
we first derive a formula for $w_{_{({\|\!\cdot\!\|}_{2},\nu)}}(A)$ in terms of ${\|A\|}_{2}$ and ${\rm tr}(A^2)$
and we then apply it to obtain several the weighted Hilbert--Schmidt numerical radius inequalities for $2\times2$ operator matrices.
In particular, we prove that
\begin{align*}
\frac{1+|1-2\nu|}{\sqrt{2}}w_{_{({\|\!\cdot\!\|}_{2},\nu)}}(A) \leq w_{_{({\|\!\cdot\!\|}_{2},\nu)}}\left(\begin{bmatrix}
0 & \mathfrak{R}_{\nu}(A) \\
\mathfrak{I}_{\nu}(A) & 0
\end{bmatrix}\right) \leq \sqrt{2}w_{_{({\|\!\cdot\!\|}_{2},\nu)}}(A).
\end{align*}
In Section 4, we give a refinement of the triangle inequality for the Hilbert--Schmidt norm as follows:
\begin{align*}
{\|A+B\|}_{2} \leq \sqrt{2w_{_{({\|\!\cdot\!\|}_{2},\nu)}}^2\left(\begin{bmatrix}
0 & A \\
B^* & 0
\end{bmatrix}\right) - (1-2\nu)^2{\|A-B\|}_{2}^2} \leq {\|A\|}_{2} + {\|B\|}_{2}.
\end{align*}
Our results extend several results in the literature.
\section{A generalization of the weighted numerical radius}
In this section, we introduce our new norm on $\mathbb{B}(\mathcal{H})$,
which generalizes the weighted numerical radius,
and present basic properties of this norm.
\begin{definition}\label{D.1.2}
Let $N(\cdot)$ be a norm on $\mathbb{B}(\mathcal{H})$ and let $0\leq \nu \leq 1$.
The function $w_{_{(N,\nu)}}(\cdot)\colon \mathbb{B}(\mathcal{H}) \rightarrow [0,+\infty)$ is defined as
\begin{align*}
w_{_{(N,\nu)}}(A)= \displaystyle{\sup_{\theta \in \mathbb{R}}}
N\left(\mathfrak{R}_{\nu}(e^{i\theta}A)\right) \qquad (A\in\mathbb{B}(\mathcal{H})).
\end{align*}
\end{definition}
\begin{remark}\label{R.2.2}
Let $N(\cdot)$ be a norm on $\mathbb{B}(\mathcal{H})$ and let $0\leq \nu \leq 1$.
For every $A\in\mathbb{B}(\mathcal{H})$, it is easy to see that
$w_{_{(N,\nu)}}(A)= \displaystyle{\sup_{\theta \in \mathbb{R}}}
N\left(\mathfrak{I}_{\nu}(ie^{i\theta}A)\right)$.
\end{remark}
\begin{remark}\label{R.3.2}
Let $N(\cdot)$ be the usual operator norm $\|\!\cdot\!\|$ and
let $0\leq \nu \leq 1$. For every $A\in\mathbb{B}(\mathcal{H})$, we have
$w_{_{(N,\nu)}}(A) = w_{_{\nu}}(A)$.
Thus $w_{_{(N,\nu)}}(\cdot)$ generalizes the weighted numerical radius $w_{_{\nu}}(\cdot)$,
which has been recently introduced in \cite{S.K.S}.
\end{remark}
\begin{remark}\label{R.4.2}
Let $N(\cdot)$ be a norm on $\mathbb{B}(\mathcal{H})$ and let $A\in\mathbb{B}(\mathcal{H})$.
Obviously, $w_{_{(N,0)}}(A) = N(A^*)$, $w_{_{(N,1)}}(A) = N(A)$ and
$w_{_{(N,1/2)}}(A)= w_{_{N}}(A)$.
Hence $w_{_{(N,\nu)}}(\cdot)$ also generalizes the numerical radius $w_{_{N}}(\cdot)$,
which has been introduced in \cite{A.K.1}.
\end{remark}
\begin{notation}\label{N.5.2}
When $N(\cdot)$ is the Hilbert--Schmidt norm ${\|\!\cdot\!\|}_{2}$, the function $w_{_{(N,\nu)}}(\cdot)$
is denoted by $w_{_{(2,\nu)}}(\cdot)$. That is,
$w_{_{(2,\nu)}}(A)= \displaystyle{\sup_{\theta \in \mathbb{R}}}
{\left\|\mathfrak{R}_{\nu}(e^{i\theta}A)\right\|}_{2}$.
\end{notation}
In the following theorem, we prove that $w_{_{(N,\nu)}}(\cdot)$ is a norm on $\mathbb{B}(\mathcal{H})$.
We use some ideas of \cite[Theorems~1,2]{A.K.1}.
\begin{theorem}\label{T.6.2}
Let $N(\cdot)$ be a norm on $\mathbb{B}(\mathcal{H})$ and let $0\leq \nu\leq 1$.
Then $w_{_{(N,\nu)}}(\cdot)$ is a norm on $\mathbb{B}(\mathcal{H})$ and the following inequalities hold for every $A\in\mathbb{B}(\mathcal{H})$:
\begin{align*}
\max\left\{\nu N(A), (1-\nu)N(A^*)\right\} \leq w_{_{(N,\nu)}}(A) \leq \max\left\{N(A), N(A^*)\right\}.
\end{align*}
\end{theorem}
\begin{proof}
Let $X\in \mathbb{B}(\mathcal{H})$. Obviously, $w_{_{(N,\nu)}}(X)\geq 0$.
Let us now suppose $w_{_{(N,\nu)}}(X) = 0$.
If $\nu=0$, then by Remark \ref{R.4.2} we get $w_{_{(N,\nu)}}(X) = N(X^*) = 0$. Thus $X^*=0$, or equivalently, $X=0$.
Hence, we may assume that $\nu\neq0$.
Then, by Definition \ref{D.1.2},
$\mathfrak{R}_{\nu}(e^{i\theta}X) = 0$ for every $\theta \in \mathbb{R}$.
Taking $\theta = 0$ and $\theta = \pi/2$, we have $\mathfrak{R}_{\nu}(X) = \mathfrak{R}_{\nu}(iX) = 0$.
Thus $X = \frac{\mathfrak{R}_{\nu}(X) -i \mathfrak{R}_{\nu}(iX)}{2\nu} = 0$.

Let $\alpha \in\mathbb{C}$. There exists $\varphi \in \mathbb{R}$ such that $\alpha= |\alpha|e^{i\varphi}$.
Hence
\begin{align*}
w_{_{(N,\nu)}}(\alpha X) &= \displaystyle{\sup_{\theta \in \mathbb{R}}}
N\left(\mathfrak{R}_{\nu}(e^{i\theta}\alpha X)\right)
\\& = \displaystyle{\sup_{\theta \in \mathbb{R}}}
N\left(\nu e^{i(\theta+\varphi)}|\alpha|X + (1-\nu)e^{-i(\theta+\varphi)}|\alpha|X^*\right)
\\& = \displaystyle{\sup_{\phi \in \mathbb{R}}}
N\left(\nu e^{i\phi}|\alpha|X + (1-\nu)e^{-i\phi}|\alpha|X^*\right)
\\& = |\alpha|\displaystyle{\sup_{\phi \in \mathbb{R}}}
N\left(\mathfrak{R}_{\nu}(e^{i\phi}X)\right) = |\alpha|w_{_{(N,\nu)}}(X).
\end{align*}
Now, let $Y, Z \in\mathbb{B}(\mathcal{H})$. We have
\begin{align*}
w_{_{(N,\nu)}}(Y+Z) &= \displaystyle{\sup_{\theta \in \mathbb{R}}}
N\left(\mathfrak{R}_{\nu}(e^{i\theta}(Y+Z))\right)
\\& = \displaystyle{\sup_{\theta \in \mathbb{R}}}
N\left(\mathfrak{R}_{\nu}(e^{i\theta}Y) + \mathfrak{R}_{\nu}(e^{i\theta}Z)\right)
\\& \leq \displaystyle{\sup_{\theta \in \mathbb{R}}}\left(
N\left(\mathfrak{R}_{\nu}(e^{i\theta}Y)\right) + N\left(\mathfrak{R}_{\nu}(e^{i\theta}Z)\right)\right)
\\& \leq \displaystyle{\sup_{\theta \in \mathbb{R}}}N\left(\mathfrak{R}_{\nu}(e^{i\theta}Y)\right)
+ \displaystyle{\sup_{\theta \in \mathbb{R}}}N\left(\mathfrak{R}_{\nu}(e^{i\theta}Z)\right)
\\& = w_{_{(N,\nu)}}(Y) + w_{_{(N,\nu)}}(Z).
\end{align*}
Thus $w_{_{(N,\nu)}}(\cdot)$ is subadditive and so $w_{_{(N,\nu)}}(\cdot)$ is a norm on $\mathbb{B}(\mathcal{H})$.

On the other hands, for every $A\in\mathbb{B}(\mathcal{H})$, by Definition \ref{D.1.2} we have
\begin{align*}
w_{_{(N,\nu)}}(A) \geq N\left(\nu e^{i\theta}A + (1-\nu)e^{-i\theta}A^*\right) \qquad (\theta \in \mathbb{R}).
\end{align*}
So, by taking $\theta = 0$ and $\theta = -\pi/2$ in the above inequality, we deduce that
\begin{align}\label{I.1.T.6.2}
w_{_{(N,\nu)}}(A) \geq N\left(\mathfrak{R}_{\nu}A\right) \quad \mbox{and} \quad  w_{_{(N,\nu)}}(A) \geq N\left(\mathfrak{I}_{\nu}A\right).
\end{align}
Since $\mathfrak{R}_{\nu}A + i\mathfrak{I}_{\nu}A = 2\nu A$, by \eqref{I.1.T.6.2} and the triangle inequality for the norm $N(\cdot)$, we have
\begin{align*}
w_{_{(N,\nu)}}(A) &\geq \frac{N\left(\mathfrak{R}_{\nu}A\right) + N\left(\mathfrak{I}_{\nu}A\right)}{2}
\\& \geq \frac{N\left(\mathfrak{R}_{\nu}A+ i\mathfrak{I}_{\nu}A\right)}{2}
= \frac{N\left(2\nu A\right)}{2} = \nu N\left(A\right),
\end{align*}
and hence
\begin{align}\label{I.2.T.6.2}
\nu N\left(A\right) \leq w_{_{(N,\nu)}}(A).
\end{align}
Replacing $A$ and $\nu$ in the above inequality by $A^*$ and $1-\nu$, respectively, we have
\begin{align}\label{I.3.T.6.2}
(1-\nu) N\left(A^*\right) \leq w_{_{(N,(1-\nu))}}(A^*).
\end{align}
Since
\begin{align}\label{I.4.T.6.2}
w_{_{(N,(1-\nu))}}(A^*) = \displaystyle{\sup_{\theta \in \mathbb{R}}}
N\left((1-\nu)e^{i\theta}A^* + \nu e^{-i\theta}A\right) = w_{_{(N,\nu)}}(A),
\end{align}
by \eqref{I.3.T.6.2}, it follows that
\begin{align}\label{I.5.T.6.2}
(1-\nu) N\left(A^*\right) \leq w_{_{(N,\nu)}}(A).
\end{align}
Now, \eqref{I.2.T.6.2} and \eqref{I.5.T.6.2} yield that
\begin{align}\label{I.6.T.6.2}
\max\left\{\nu N(A), (1-\nu)N(A^*)\right\} \leq w_{_{(N,\nu)}}(A).
\end{align}
Furthermore, by the triangle inequality for the norm $N(\cdot)$, we have
\begin{align*}
w_{_{(N,\nu)}}(A)&= \displaystyle{\sup_{\theta \in \mathbb{R}}}
N\left(\nu e^{i\theta}A + (1-\nu)e^{-i\theta}A^*\right)
\\& \leq \nu N\left(A\right) + (1-\nu)N\left(A^*\right)
\\& \leq \nu\max\left\{N(A), N(A^*)\right\} + (1-\nu)\max\left\{N(A), N(A^*)\right\}
\\& = \max\left\{N(A), N(A^*)\right\},
\end{align*}
and so
\begin{align}\label{I.7.T.6.2}
w_{_{(N,\nu)}}(A) \leq \max\left\{N(A), N(A^*)\right\}.
\end{align}
From \eqref{I.6.T.6.2} and \eqref{I.7.T.6.2}, we deduce the desired result.
\end{proof}
In the following result we state some properties of the norm $w_{_{(N,\nu)}}(\cdot)$.
\begin{proposition}\label{P.7.2}
Let $N(\cdot)$ be a norm on $\mathbb{B}(\mathcal{H})$ and let $0\leq \nu\leq 1$.
For every $A\in\mathbb{B}(\mathcal{H})$, the following properties hold:
\begin{itemize}
\item[(i)] $w_{_{(N,\nu)}}(A^*) = w_{_{(N,1-\nu)}}(A)$.
\item[(ii)] If $A$ is self-adjoint, then $w_{_{(N,\nu)}}(A) = N(A)$.
\item[(iii)] $w_{_{(N,\nu)}}(A)= \frac{1}{2}\displaystyle{\sup_{\theta,\varphi \in \mathbb{R}}}
N\left(\mathfrak{R}_{\nu}((e^{i\theta}-ie^{i\varphi})A)\right)$.
\end{itemize}
\end{proposition}
\begin{proof}
(i) It follows immediately from \eqref{I.4.T.6.2}.

(ii) Since $A$ is self-adjoint, we have $\mathfrak{R}_{\nu}(e^{i\theta}A) = \left(\nu e^{i\theta} +(1-\nu)e^{-i\theta}\right)A$.
Thus
\begin{align*}
w_{_{(N,\nu)}}(A)& = \displaystyle{\sup_{\theta \in \mathbb{R}}}
N\left(\mathfrak{R}_{\nu}(e^{i\theta})A\right)
\\& = \displaystyle{\sup_{\theta \in \mathbb{R}}}
\left|\nu e^{i\theta} +(1-\nu)e^{-i\theta}\right|N(A)
\\& = N(A)\displaystyle{\sup_{\theta \in \mathbb{R}}}
\sqrt{1-4\nu(1-\nu)\sin^2\theta} = N(A).
\end{align*}

(iii)
We have
\begin{align*}
w_{_{(N,\nu)}}(A)&= \frac{1}{2}\displaystyle{\sup_{\theta \in \mathbb{R}}}
N\left(\mathfrak{R}_{\nu}(e^{i\theta}A) + \mathfrak{R}_{\nu}(e^{i\theta}A)\right)
\\& = \frac{1}{2}\displaystyle{\sup_{\theta \in \mathbb{R}}}
N\left(\mathfrak{R}_{\nu}(e^{i\theta}A) + \mathfrak{I}_{\nu}(e^{i(\theta+\pi/2)}A)\right)
\\&\leq \frac{1}{2}\displaystyle{\sup_{\theta,\varphi \in \mathbb{R}}}
N\left(\mathfrak{R}_{\nu}(e^{i\theta}A) + \mathfrak{I}_{\nu}(e^{i\varphi}A)\right)
\\&= \frac{1}{2}\displaystyle{\sup_{\theta,\varphi \in \mathbb{R}}}
N\left(\nu\left(e^{i\theta} -ie^{i\varphi}\right)A + (1-\nu)\left(\left(e^{i\theta} -ie^{i\varphi}\right)A\right)^*\right)
\\&= \frac{1}{2}\displaystyle{\sup_{\theta,\varphi \in \mathbb{R}}}
N\left(\mathfrak{R}_{\nu}\left(\left(e^{i\theta} -ie^{i\varphi}\right)A\right)\right)
\\&\leq \frac{1}{2}\displaystyle{\sup_{\theta,\varphi \in \mathbb{R}}}
w_{_{(N,\nu)}}\left(\left(\left(e^{i\theta} -ie^{i\varphi}\right)A\right)\right)
\\&= \frac{1}{2}\displaystyle{\sup_{\theta,\varphi \in \mathbb{R}}}
\left|e^{i\theta} -ie^{i\varphi}\right|w_{_{(N,\nu)}}(A)
\\&= \frac{w_{_{(N,\nu)}}(A)}{2}\displaystyle{\sup_{\theta,\varphi \in \mathbb{R}}}
\sqrt{2-2\sin(\theta-\varphi)}
= w_{_{(N,\nu)}}(A),
\end{align*}
and so $w_{_{(N,\nu)}}(A)= \frac{1}{2}\displaystyle{\sup_{\theta,\varphi \in \mathbb{R}}}
N\left(\mathfrak{R}_{\nu}((e^{i\theta}-ie^{i\varphi})A)\right)$.
\end{proof}
We now derive some properties of the norm $w_{_{(N,\nu)}}(\cdot)$ when $N(\cdot)$ is a self-adjoint norm on $\mathbb{B}(\mathcal{H})$.
\begin{proposition}\label{P.8.2}
Let $N(\cdot)$ be a self-adjoint norm on $\mathbb{B}(\mathcal{H})$ and let $0\leq \mu, \nu\leq 1$.
Then
\begin{itemize}
\item[(i)] $w_{_{(N,\nu)}}(\cdot)$ is self-adjoint.
\item[(ii)] $w_{_{(N,\nu)}}(\cdot)$ is equivalent to $N(\cdot)$ and the following inequalities hold for
every $A\in\mathbb{B}(\mathcal{H})$:
\begin{align*}
\frac{1+|1-2\nu|}{2}N(A)\leq w_{_{(N,\nu)}}(A) \leq N(A).
\end{align*}
\item[(iii)] For every $A\in\mathbb{B}(\mathcal{H})$, the function $f(\nu) = w_{_{(N,\nu)}}(A)$ is a convex
continuous function on [0, 1] that attains its minimum at $\nu = 1/2$ and its maximum at $\nu = 0, \nu = 1$.
\item[(iv)] For every $A\in\mathbb{B}(\mathcal{H})$, $w_{_{(N,\nu)}}(A)\leq w_{_{(N,\mu)}}(A)$ if and only if $|\nu -1/2|\leq|\mu -1/2|$.
\end{itemize}
\end{proposition}
\begin{proof}
(i) Let $A\in\mathbb{B}(\mathcal{H})$. Since the norm $N(\cdot)$ is self-adjoint, we have
\begin{align*}
w_{_{(N,\nu)}}(A^*)&= \displaystyle{\sup_{\theta \in \mathbb{R}}}
N\left(\nu e^{i\theta}A^* + (1-\nu)e^{-i\theta}A\right)
\\&= \displaystyle{\sup_{\theta \in \mathbb{R}}}
N\left(\nu e^{-i\theta}A + (1-\nu)e^{i\theta}A^*\right) = w_{_{(N,\nu)}}(A).
\end{align*}
(ii) Since $N(A^*)=N(A)$ for every $A\in\mathbb{B}(\mathcal{H})$ and $\max\{\nu, 1-\nu\} = \frac{1+|1-2\nu|}{2}$,
the desired inequalities follow from Theorem \ref{T.6.2}.

(iii) Let $A\in\mathbb{B}(\mathcal{H})$ and let $0\leq t\leq1$. The following proof is a modification of the one given
by Sheikhhosseini et al. \cite[Theorem~2.5]{S.K.S}.
We have
\begin{align*}
f\left(t\nu +(1-t)\mu\right) &= w_{_{(N,\left(t\nu +(1-t)\mu\right))}}(A)
\\& = \displaystyle{\sup_{\theta \in \mathbb{R}}}
N\left(\mathfrak{R}_{\left(t\nu +(1-t)\mu\right)}(e^{i\theta}A)\right)
\\& = \displaystyle{\sup_{\theta \in \mathbb{R}}}
N\left(t\mathfrak{R}_{\nu}(e^{i\theta}A) + (1-t)\mathfrak{R}_{\mu}(e^{i\theta}A)\right)
\\& \leq \displaystyle{\sup_{\theta \in \mathbb{R}}}
\Big(tN\left(\mathfrak{R}_{\nu}(e^{i\theta}A)\right) + (1-t)N\left(\mathfrak{R}_{\mu}(e^{i\theta}A)\right)\Big)
\\& \leq tw_{_{(N,\nu)}}(A) + (1-t)w_{_{(N,\mu)}}(A) = tf(\nu) + (1-t)f(\mu),
\end{align*}
and hence $f\left(t\nu +(1-t)\mu\right) \leq tf(\nu) + (1-t)f(\mu)$.
Therefore $f$ is convex on $[0, 1]$ and so $f$ is continuous on $(0, 1)$.
Also, by (ii), we have
\begin{align*}
0 \leq N(A) - w_{_{(N,\nu)}}(A) \leq \frac{1-|1-2\nu|}{2}N(A).
\end{align*}
Thus $f$ is continuous at $\nu=0$ and $\nu=1$. Hence $f$ is continuous on $[0, 1]$.
By (i) and Proposition \ref{P.7.2}(i), we have $f(\nu) = f(1-\nu)$. Thus $f$ is symmetric about $\nu =1/2$.
This shows that $f$ attains its minimum at $\nu =1/2$ and its maximum at
$\nu = 0, \nu = 1$.

(iv) From (iii), the function $f(\nu) = w_{_{(N,\nu)}}(A)$ is convex on [0, 1] and symmetric around $\nu =1/2$, it follows
that it is decreasing on $[0, 1/2]$ and increasing on $[1/2, 1]$. Now, (iv) is trivial.
\end{proof}
Our next result reads as follows.
\begin{proposition}\label{P.9.2}
Let $N(\cdot)$ be a norm on $\mathbb{B}(\mathcal{H})$ and let $0\leq \nu\leq 1$.
\begin{itemize}
\item[(i)] If $N(\cdot)$ is weakly unitarily invariant, then so is $w_{_{(N,\nu)}}(\cdot)$.
\item[(ii)] If $N(\cdot)$ is unitarily invariant, then $w_{_{(N,\nu)}}(A) \leq w_{\nu}(A)N(I)$.
\end{itemize}
\end{proposition}
\begin{proof}
(i) Let $N(\cdot)$ be weakly unitarily invariant and let $A\in\mathbb{B}(\mathcal{H})$.
Let $U\in\mathbb{B}(\mathcal{H})$ be unitary. We have
\begin{align*}
w_{_{(N,\nu)}}(U^*AU) &= \displaystyle{\sup_{\theta \in \mathbb{R}}}
N\left(\mathfrak{R}_{\nu}(e^{i\theta}U^*AU)\right)
\\& = \displaystyle{\sup_{\theta \in \mathbb{R}}}
N\left(U^*\mathfrak{R}_{\nu}(e^{i\theta}A)U\right)
\\& = \displaystyle{\sup_{\theta \in \mathbb{R}}}
N\left(\mathfrak{R}_{\nu}(e^{i\theta}A)\right) = w_{_{(N,\nu)}}(A).
\end{align*}
Thus $w_{_{(N,\nu)}}(\cdot)$ is weakly unitarily invariant.

(ii) Let $N(\cdot)$ be unitarily invariant and let $A\in\mathbb{B}(\mathcal{H})$. Let $\epsilon>0$.
For every $\theta \in \mathbb{R}$, since
$\left\|\frac{\mathfrak{R}_{\nu}(e^{i\theta}A)}{\left\|\mathfrak{R}_{\nu}(e^{i\theta}A)\right\| + \epsilon}\right\| < 1$,
by Gardner's theorem (see, e.g., \cite{Black}), there exist unitaries $U_1, U_2, \cdots U_n\in\mathbb{B}(\mathcal{H})$ such that
\begin{align*}
\frac{\mathfrak{R}_{\nu}(e^{i\theta}A)}{\left\|\mathfrak{R}_{\nu}(e^{i\theta}A)\right\| + \epsilon}
= \frac{U_1+ U_2+ \cdots +U_n}{n}.
\end{align*}
Since the norm $N(\cdot)$ is unitarily invariant, we have $N\left(U_k\right) = N\left(I\right)$ for all $1\leq k\leq n$,
and therefore,
\begin{align*}
N\left(\frac{\mathfrak{R}_{\nu}(e^{i\theta}A)}{\left\|\mathfrak{R}_{\nu}(e^{i\theta}A)\right\| + \epsilon}\right)
&= \frac{N\left(U_1+ U_2+ \cdots +U_n\right)}{n}
\\& \leq \frac{N\left(U_1\right)+ N\left(U_2\right)+ \cdots +N\left(U_n\right)}{n} = N\left(I\right).
\end{align*}
From this it follows that
\begin{align*}
N\left(\mathfrak{R}_{\nu}(e^{i\theta}A)\right)
\leq \left(\left\|\mathfrak{R}_{\nu}(e^{i\theta}A)\right\| + \epsilon\right)N\left(I\right).
\end{align*}
Taking the supremum over $\theta \in \mathbb{R}$ in the above inequality, we get
\begin{align*}
w_{_{(N,\nu)}}(A) \leq \left(w_{\nu}(A) + \epsilon\right)N\left(I\right).
\end{align*}
Now, letting $\epsilon \rightarrow 0^{+}$, we conclude $w_{_{(N,\nu)}}(A) \leq w_{\nu}(A)N(I)$.
\end{proof}
\begin{remark}\label{R.10.2}
Since the Hilbert--Schmidt norm ${\|\cdot\|}_{2}$ is unitarily invariant and
${\|I\|}_{2} = 1$, by Proposition \ref{P.9.2}(ii) for every $A\in \mathcal{C}_2(\mathcal{H})$, we have $w_{_{(2,\nu)}}(A) \leq w_{\nu}(A)$.
\end{remark}
\section{The weighted Hilbert--Schmidt numerical radius inequalities for operator matrices}
In this section, we first derive a formula for $w_{_{(2,\nu)}}(A)$ in terms of ${\|A\|}_{2}$ and ${\rm tr}(A^2)$
and we then apply it to obtain several the weighted Hilbert--Schmidt numerical radius inequalities for $2\times2$ operator matrices.
These inequalities generalize known Hilbert--Schmidt numerical radius inequalities (see, e.g., \cite{Bak.S, Ba.K, H.K.S, M.S, N.S.M, Sh, X.Y.Z}).
\begin{theorem}\label{T.1.3}
Let $A\in \mathcal{C}_2(\mathcal{H})$ and let $0\leq \nu \leq 1$. Then
\begin{align*}
w_{_{(2,\nu)}}^2(A) = (2\nu^2-2\nu+1){\|A\|}_{2}^2 + 2\nu(1-\nu)\left|{\rm tr}(A^2)\right|.
\end{align*}
\end{theorem}
\begin{proof}
For every $\theta \in \mathbb{R}$, we have
\begingroup\makeatletter\def\f@size{10}\check@mathfonts
\begin{align*}
{\left\|\mathfrak{R}_{\nu}(ie^{i\theta}A)\right\|}_{2}^2
&= {\rm tr}\left(\left(\mathfrak{R}_{\nu}(ie^{i\theta}A)\right)^*\mathfrak{R}_{\nu}(ie^{i\theta}A)\right)
\\& = {\rm tr}\Big(\left(\nu e^{-i\theta}A^* + (1-\nu)e^{i\theta}A\right)\left(\nu e^{i\theta}A + (1-\nu)e^{-i\theta}A^*\right)\Big)
\\& = {\rm tr}\Big(\nu^2 A^*A + \nu(1-\nu)e^{2i\theta}A^2 + e^{-2i\theta}\nu(1-\nu)(A^*)^2 + (1-\nu)^2AA^*\Big)
\\& = \nu^2{\rm tr}\left(A^*A\right) + \nu(1-\nu)\left(e^{2i\theta}{\rm tr}\left(A^2\right) + e^{-2i\theta}{\rm tr}\left((A^2)^*\right)\right)
+ (1-\nu)^2{\rm tr}\left(AA^*\right)
\\& = \nu^2{\|A\|}_{2}^2 +2\nu(1-\nu)\mathfrak{R}\left(e^{2i\theta}{\rm tr}\left(A^2\right)\right) + (1-\nu)^2{\|A\|}_{2}^2
\\& = (2\nu^2-2\nu+1){\|A\|}_{2}^2 +2\nu(1-\nu)\mathfrak{R}\left(e^{2i\theta}{\rm tr}\left(A^2\right)\right).
\end{align*}
\endgroup
Therefore,
\begin{align*}
w_{_{(2,\nu)}}^2(A) &= \displaystyle{\sup_{\theta \in \mathbb{R}}}
{\left\|\mathfrak{R}_{\nu}(e^{i\theta}A)\right\|}_{2}^2
\\& = \displaystyle{\sup_{\theta \in \mathbb{R}}}
\Big((2\nu^2-2\nu+1){\|A\|}_{2}^2 +2\nu(1-\nu)\mathfrak{R}\left(e^{2i\theta}{\rm tr}\left(A^2\right)\right)\Big)
\\& = (2\nu^2-2\nu+1){\|A\|}_{2}^2 + 2\nu(1-\nu)\left|{\rm tr}(A^2)\right|.
\end{align*}
\end{proof}
\begin{remark}
Theorem \ref{T.1.3} is a generalization of a result due to Abu-Omar and Kittaneh \cite{A.K.1}.
In fact if $\nu = 1/2$, then
\begin{align*}
w_{_{2}}^2(A) = \frac{1}{2}{\|A\|}_{2}^2 + \frac{1}{2}\left|{\rm tr}(A^2)\right|
\end{align*}
which has been proven in \cite[Theorem~7]{A.K.1}.
\end{remark}
\begin{theorem}\label{T.2.3}
Let $A\in \mathcal{C}_2(\mathcal{H})$ and let $0\leq \nu \leq 1$. Then
\begin{align*}
\sqrt{2\nu^2-2\nu+1}{\|A\|}_{2} \leq w_{_{(2,\nu)}}(A) \leq {\|A\|}_{2}.
\end{align*}
\end{theorem}
\begin{proof}
The proof follows from Theorem \ref{T.1.3} and the fact that $0\leq \left|{\rm tr}(A^2)\right| \leq {\|A\|}_{2}^2$.
\end{proof}
\begin{remark}\label{R.3.3}
Let $0\leq \nu \leq 1$. It is easy to see that $\frac{1+|1-2\nu|}{2} \leq \sqrt{2\nu^2-2\nu+1}$, thus
the first inequality in Theorem \ref{T.2.3} refines the first inequality in Proposition \ref{P.8.2}(ii).
\end{remark}
\begin{corollary}\label{C.4.3}
Let $A\in \mathcal{C}_2(\mathcal{H})$ and let $0< \nu < 1$.
Then the following conditions are equivalent:
\begin{itemize}
\item[(i)] $w_{_{(2,\nu)}}(A) = \sqrt{2\nu^2-2\nu+1}{\|A\|}_{2}$.
\item[(ii)] ${\rm tr}(A^2) =0$.
\end{itemize}
\end{corollary}
\begin{proof}
The proof follows immediately from Theorem \ref{T.1.3}.
\end{proof}
\begin{corollary}\label{C.5.3}
Let $A\in \mathcal{C}_2(\mathcal{H})$ and let $0< \nu < 1$.
Then the following conditions are equivalent:
\begin{itemize}
\item[(i)] $w_{_{(2,\nu)}}(A) = {\|A\|}_{2}$.
\item[(ii)] $A$ is normal and the squares of its
nonzero eigenvalues have the same argument.
\end{itemize}
\end{corollary}
\begin{proof}
By using Theorem \ref{T.1.3} and applying a similar proof as in of \cite[Corollary~2(ii)]{A.K.1},
the required result is obtained.
\end{proof}
The following lemma is known in the literature (see, e.g., \cite{Bhatia}).
\begin{lemma}\label{L.6.3}
Let $X, Y, Z, W \in \mathcal{C}_2(\mathcal{H})$. Then
\begin{align*}
{\left\|\begin{bmatrix}
X & Y \\
Z & W
\end{bmatrix}\right\|}_{2}^2 = {\|X\|}_{2}^2 + {\|Y\|}_{2}^2 + {\|Z\|}_{2}^2 + {\|W\|}_{2}^2.
\end{align*}
\end{lemma}
In the next theorem we compute the weighted Hilbert--Schmidt numerical
radius for certain $2\times2$ operator matrices defined on $\mathcal{H}\oplus \mathcal{H}$.
\begin{theorem}\label{T.7.3}
Let $A, B, C, D \in \mathcal{C}_2(\mathcal{H})$ and let $0\leq \nu \leq 1$ and $\theta \in \mathbb{R}$.
\begin{itemize}
\item[(i)] $w_{_{(2,\nu)}}\left(\begin{bmatrix}
0 & A \\
B & 0
\end{bmatrix}\right) = w_{_{(2,\nu)}}\left(\begin{bmatrix}
0 & B \\
A & 0
\end{bmatrix}\right)$.
\item[(ii)] $w_{_{(2,\nu)}}\left(\begin{bmatrix}
0 & A \\
B & 0
\end{bmatrix}\right) = w_{_{(2,\nu)}}\left(\begin{bmatrix}
0 & A \\
e^{i\theta}B & 0
\end{bmatrix}\right)$.
\item[(iii)] $w_{_{(2,\nu)}}\left(\begin{bmatrix}
A & A \\
-A & -A
\end{bmatrix}\right) = 2\sqrt{2\nu^2-2\nu+1}{\|A\|}_{2}$.
\item[(iv)] $w_{_{(2,\nu)}}\left(\begin{bmatrix}
0 & A \\
A & 0
\end{bmatrix}\right) = \sqrt{2}w_{_{(2,\nu)}}(A)$.
\item[(v)] $w_{_{(2,\nu)}}^2\left(\begin{bmatrix}
A & B \\
0 & 0
\end{bmatrix}\right) = w_{_{(2,\nu)}}^2(A) + (2\nu^2-2\nu+1){\|B\|}_{2}^2$.
\item[(vi)]
If ${\rm tr}(BC) = 0$, then
$w_{_{(2,\nu)}}^2\left(\begin{bmatrix}
A & B \\
C & iA
\end{bmatrix}\right) = (2\nu^2-2\nu+1)\left(2{\|A\|}_{2}^2 + {\|B\|}_{2}^2+{\|C\|}_{2}^2\right)$.
\item[(vii)] If $A, B$ are positive, then
$w_{_{(2,\nu)}}^2\left(\begin{bmatrix}
0 & A \\
B & 0
\end{bmatrix}\right) = \frac{{\|A+B\|}_{2}^2 + (1-2\nu)^2{\|A-B\|}_{2}^2}{2}$.
\item[(viii)] If $A, B$ are self-adjoint, then
$w_{_{(2,\nu)}}^2\left(\begin{bmatrix}
A & 0 \\
0 & B
\end{bmatrix}\right) = w_{_{(2,\nu)}}^2(A) + w_{_{(2,\nu)}}^2(B)$.
\item[(ix)] If $A, B$ are self-adjoint, then
\begin{align*}
w_{_{(2,\nu)}}^2\left(\begin{bmatrix}
A & B \\
B & A
\end{bmatrix}\right) = w_{_{(2,\nu)}}^2(A+B) + w_{_{(2,\nu)}}^2(A-B).
\end{align*}
\end{itemize}
\end{theorem}
\begin{proof}
(i)
Let $U = \begin{bmatrix}
0 & I \\
I & 0
\end{bmatrix}$. Then $U$ is a unitary operator on $\mathcal{H}\oplus \mathcal{H}$ and by Proposition \ref{P.9.2}(i) we have
\begin{align*}
w_{_{(2,\nu)}}\left(\begin{bmatrix}
0 & A \\
B & 0
\end{bmatrix}\right) = w_{_{(2,\nu)}}\left(U\begin{bmatrix}
0 & A \\
B & 0
\end{bmatrix}U^*\right) = w_{_{(2,\nu)}}\left(\begin{bmatrix}
0 & B \\
A & 0
\end{bmatrix}\right).
\end{align*}
(ii) It is similar than (i), only we use $\begin{bmatrix}
I & 0 \\
0 & e^{i\theta/2}I
\end{bmatrix}$ instead of $\begin{bmatrix}
0 & I \\
I & 0
\end{bmatrix}$.

(iii)
Let $U = \frac{1}{\sqrt{2}}\begin{bmatrix}
I & I \\
-I & I
\end{bmatrix}$. Then $U$ is a unitary operator on $\mathcal{H}\oplus \mathcal{H}$.
Since
${\begin{bmatrix}
A & A \\
-A & -A
\end{bmatrix}}^2 = 0$, by Theorem \ref{T.1.3}, we have
\begin{align*}
w_{_{(2,\nu)}}^2\left(\begin{bmatrix}
A & A \\
-A & -A
\end{bmatrix}\right) = (2\nu^2-2\nu+1){\left\|\begin{bmatrix}
A & A \\
-A & -A
\end{bmatrix}\right\|}_{2}^2.
\end{align*}
Also, since the Hilbert--Schmidt norm ${\|\cdot\|}_{2}$ is unitarily invariant, we have
\begin{align*}
{\left\|\begin{bmatrix}
A & A \\
-A & -A
\end{bmatrix}\right\|}_{2} =
{\left\|U\begin{bmatrix}
A & A \\
-A & -A
\end{bmatrix}U^*\right\|}_{2}
= 2{\left\|\begin{bmatrix}
0 & 0 \\
A & 0
\end{bmatrix}\right\|}_{2}.
\end{align*}
Now, utilizing Lemma \ref{L.6.3}, we deduce the desired result.

(iv) From Theorem \ref{T.1.3} and Lemma \ref{L.6.3} it follows that
\begin{align*}
w_{_{(2,\nu)}}^2\left(\begin{bmatrix}
0 & A \\
A & 0
\end{bmatrix}\right) &= (2\nu^2-2\nu+1){\left\|\begin{bmatrix}
0 & A \\
A & 0
\end{bmatrix}\right\|}_{2}^2 + 2\nu(1-\nu)\left|{\rm tr}\left({\begin{bmatrix}
0 & A \\
A & 0
\end{bmatrix}}^2\right)\right|
\\& = (2\nu^2-2\nu+1)\left({\|A\|}_{2}^2 + {\|A\|}_{2}^2\right) + 2\nu(1-\nu)\left|{\rm tr}\left(\begin{bmatrix}
A^2 & 0 \\
0 & A^2
\end{bmatrix}\right)\right|
\\& = 2\Big((2\nu^2-2\nu+1){\|A\|}_{2}^2 + 2\nu(1-\nu)\left|{\rm tr}A^2\right|\Big)
= 2w_{_{(2,\nu)}}^2(A),
\end{align*}
which gives (iv).

(v) By Theorem \ref{T.1.3} and Lemma \ref{L.6.3}, we have
\begin{align*}
w_{_{(2,\nu)}}^2\left(\begin{bmatrix}
A & B \\
0 & 0
\end{bmatrix}\right) &= (2\nu^2-2\nu+1){\left\|\begin{bmatrix}
A & B \\
0 & 0
\end{bmatrix}\right\|}_{2}^2 + 2\nu(1-\nu)\left|{\rm tr}\left({\begin{bmatrix}
A & B \\
0 & 0
\end{bmatrix}}^2\right)\right|
\\& = (2\nu^2-2\nu+1)\left({\left\|A\right\|}_{2}^2 + {\left\|B\right\|}_{2}^2\right) + 2\nu(1-\nu)\left|{\rm tr}\left(\begin{bmatrix}
A^2 & AB \\
0 & 0
\end{bmatrix}\right)\right|
\\& = (2\nu^2-2\nu+1){\|A\|}_{2}^2 + (2\nu^2-2\nu+1){\|B\|}_{2}^2 + 2\nu(1-\nu)\left|{\rm tr}(A^2)\right|
\\& = w_{_{(2,\nu)}}^2(A) + (2\nu^2-2\nu+1){\|B\|}_{2}^2.
\end{align*}

(vi) Let ${\rm tr}(BC) = 0$. We have
\begin{align*}
\left|{\rm tr}\left({\begin{bmatrix}
A & B \\
C & iA
\end{bmatrix}}^2\right)\right| = \left|{\rm tr}\left({\begin{bmatrix}
A^2+BC & AB+iBA \\
CA+iAC & CB-A^2
\end{bmatrix}}^2\right)\right|
= 2\left|{\rm tr}(BC)\right| = 0.
\end{align*}
Also, by Lemma \ref{L.6.3}, we get
${\left\|\begin{bmatrix}
A & B \\
C & iA
\end{bmatrix}\right\|}_{2}^2 = 2{\|A\|}_{2}^2 + {\|B\|}_{2}^2+{\|C\|}_{2}^2$.
So, by Theorem \ref{T.1.3}, we deduce the desired result.

(vii) Let $A, B$ be positive. Then by Theorem \ref{T.1.3} and Lemma \ref{L.6.3} it follows that
\begin{align*}
w_{_{(2,\nu)}}^2\left(\begin{bmatrix}
0 & A \\
B & 0
\end{bmatrix}\right) &= (2\nu^2-2\nu+1){\left\|\begin{bmatrix}
0 & A \\
B & 0
\end{bmatrix}\right\|}_{2}^2 + 2\nu(1-\nu)\left|{\rm tr}\left({\begin{bmatrix}
0 & A \\
B & 0
\end{bmatrix}}^2\right)\right|
\\& = (2\nu^2-2\nu+1)\left({\|A\|}_{2}^2 + {\|B\|}_{2}^2\right) + 2\nu(1-\nu)\left|{\rm tr}\left(\begin{bmatrix}
AB & 0 \\
0 & BA
\end{bmatrix}\right)\right|
\\& = (2\nu^2-2\nu+1)\left({\|A\|}_{2}^2 + {\|B\|}_{2}^2\right) + 4\nu(1-\nu)\left|{\rm tr}\left(AB\right)\right|
\\& = (2\nu^2-2\nu+1)\left({\rm tr}\left(A^2\right) + {\rm tr}\left(B^2\right)\right) + 4\nu(1-\nu){\rm tr}\left(AB\right)
\\& = \frac{{\rm tr}\left((A+B)^2\right) + (1-2\nu)^2{\rm tr}\left((A-B)^2\right)}{2}
\\& = \frac{{\|A+B\|}_{2}^2 + (1-2\nu)^2{\|A-B\|}_{2}^2}{2}.
\end{align*}
Hence $w_{_{(2,\nu)}}^2\left(\begin{bmatrix}
0 & A \\
B & 0
\end{bmatrix}\right) = \frac{{\|A+B\|}_{2}^2 + (1-2\nu)^2{\|A-B\|}_{2}^2}{2}$.

(viii) Let $A, B$ be self-adjoint. Then $A^2, B^2$ are positive and so
\begin{align*}
\left|{\rm tr}(A^2 + B^2)\right| = \left|{\rm tr}(A^2)\right| + \left|{\rm tr}(B^2)\right|.
\end{align*}
Now, by Theorem \ref{T.1.3} and Lemma \ref{L.6.3}, we have
\begin{align*}
w_{_{(2,\nu)}}^2\left(\begin{bmatrix}
A & 0 \\
0 & B
\end{bmatrix}\right) &= (2\nu^2-2\nu+1){\left\|\begin{bmatrix}
A & 0 \\
0 & B
\end{bmatrix}\right\|}_{2}^2 + 2\nu(1-\nu)\left|{\rm tr}\left({\begin{bmatrix}
A & 0 \\
0 & B
\end{bmatrix}}^2\right)\right|
\\& = (2\nu^2-2\nu+1)\left({\|A\|}_{2}^2 + {\|B\|}_{2}^2\right) + 2\nu(1-\nu)\left|{\rm tr}\left(\begin{bmatrix}
A^2 & 0 \\
0 & B^2
\end{bmatrix}\right)\right|
\\& = (2\nu^2-2\nu+1)\left({\|A\|}_{2}^2 + {\|B\|}_{2}^2\right) + 2\nu(1-\nu)\left(\left|{\rm tr}(A^2)\right| + \left|{\rm tr}(B^2)\right|\right)
\\& = w_{_{(2,\nu)}}^2(A) + w_{_{(2,\nu)}}^2(B).
\end{align*}
(ix) Let $A, B$ be self-adjoint and let
$U = \frac{1}{\sqrt{2}}\begin{bmatrix}
I & I \\
-I & I
\end{bmatrix}$.
So, by Proposition \ref{P.9.2}(i) and (viii), we obtain
\begin{align*}
w_{_{(2,\nu)}}^2\left(\begin{bmatrix}
A & B \\
B & A
\end{bmatrix}\right) &= w_{_{(2,\nu)}}^2\left(U\begin{bmatrix}
A & B \\
B & A
\end{bmatrix}U^*\right)
\\& = w_{_{(2,\nu)}}^2\left(\begin{bmatrix}
A+B & 0 \\
0 & A-B
\end{bmatrix}\right)
\\& = w_{_{(2,\nu)}}^2(A+B) + w_{_{(2,\nu)}}^2(A-B).
\end{align*}
\end{proof}
In the following theorem we give an upper bound for the weighted Hilbert--Schmidt numerical radius of the
general $2\times2$ operator matrix
$\begin{bmatrix}
A & B \\
C & D
\end{bmatrix}$.
\begin{theorem}\label{T.8.3}
Let $A, B, C, D \in \mathcal{C}_2(\mathcal{H})$ and let $0\leq \nu \leq 1$. Then
\begin{align*}
w_{_{(2,\nu)}}^2\left(\begin{bmatrix}
A & B \\
C & D
\end{bmatrix}\right) \leq w_{_{(2,\nu)}}^2(A) + w_{_{(2,\nu)}}^2(D) + {\|B\|}_{2}^2 + {\|C\|}_{2}^2.
\end{align*}
\end{theorem}
\begin{proof}
By Theorem \ref{T.1.3}, Lemma \ref{L.6.3} and the triangle inequality, we have
\begingroup\makeatletter\def\f@size{9}\check@mathfonts
\begin{align}\label{I.1.T.8.3}
w_{_{(2,\nu)}}^2\left(\begin{bmatrix}
A & B \\
C & D
\end{bmatrix}\right) &= (2\nu^2-2\nu+1){\left\|\begin{bmatrix}
A & B \\
C & D
\end{bmatrix}\right\|}_{2}^2 + 2\nu(1-\nu)\left|{\rm tr}\left({\begin{bmatrix}
A & B \\
C & D
\end{bmatrix}}^2\right)\right|\nonumber
\\& = (2\nu^2-2\nu+1)\left({\|A\|}_{2}^2 + {\|B\|}_{2}^2 + {\|C\|}_{2}^2 + {\|D\|}_{2}^2\right)\nonumber
\\& \qquad \qquad + 2\nu(1-\nu)\left|{\rm tr}\left(A^2\right) + {\rm tr}\left(D^2\right) + 2 {\rm tr}\left(BC\right)\right|\nonumber
\\& \leq (2\nu^2-2\nu+1){\|A\|}_{2}^2 + 2\nu(1-\nu)\left|{\rm tr}\left(A^2\right)\right|\nonumber
\\& \qquad \qquad + (2\nu^2-2\nu+1){\|D\|}_{2}^2 + 2\nu(1-\nu)\left|{\rm tr}\left(D^2\right)\right|\nonumber
\\& \qquad \qquad \qquad + (2\nu^2-2\nu+1)\left({\|B\|}_{2}^2 + {\|C\|}_{2}^2\right) + 4\nu(1-\nu)\left|{\rm tr}\left(BC\right)\right|\nonumber
\\& = w_{_{(2,\nu)}}^2(A) + w_{_{(2,\nu)}}^2(D)
+ (2\nu^2-2\nu+1)\left({\|B\|}_{2}^2 + {\|C\|}_{2}^2\right) + 4\nu(1-\nu)\left|{\rm tr}\left(BC\right)\right|.
\end{align}
\endgroup
Hence, by \eqref{I.1.T.8.3}, \eqref{I.12.1} and the arithmetic-geometric mean inequality, we conclude
\begingroup\makeatletter\def\f@size{9}\check@mathfonts
\begin{align*}
w_{_{(2,\nu)}}^2\left(\begin{bmatrix}
A & B \\
C & D
\end{bmatrix}\right) & \leq w_{_{(2,\nu)}}^2(A) + w_{_{(2,\nu)}}^2(D)
+ (2\nu^2-2\nu+1)\left({\|B\|}_{2}^2 + {\|C\|}_{2}^2\right) + 4\nu(1-\nu){\|B\|}_{2}{\|C\|}_{2}
\\& \leq w_{_{(2,\nu)}}^2(A) + w_{_{(2,\nu)}}^2(D)
+ (2\nu^2-2\nu+1)\left({\|B\|}_{2}^2 + {\|C\|}_{2}^2\right) + 2\nu(1-\nu)\left({\|B\|}_{2}^2+{\|C\|}_{2}^2\right)
\\& = w_{_{(2,\nu)}}^2(A) + w_{_{(2,\nu)}}^2(D) + {\|B\|}_{2}^2 + {\|C\|}_{2}^2.
\end{align*}
\endgroup
\end{proof}
\begin{remark}\label{R.9.3}
Let $A, B, C, D\in \mathcal{C}_2(\mathcal{H})$ and let $0\leq \nu \leq 1$. If ${\rm tr}(BC) = 0$, then
by the inequality \eqref{I.1.T.8.3} we have
\begin{align*}
w_{_{(2,\nu)}}^2\left(\begin{bmatrix}
A & B \\
C & D
\end{bmatrix}\right) \leq w_{_{(2,\nu)}}^2(A) + w_{_{(2,\nu)}}^2(D) + (2\nu^2-2\nu+1)\left({\|B\|}_{2}^2 + {\|C\|}_{2}^2\right).
\end{align*}
\end{remark}
As an immediate consequence of Theorem \ref{T.8.3}, we have the following result.
\begin{corollary}\label{C.10.3}
Let $A, B\in \mathcal{C}_2(\mathcal{H})$ and let $0\leq \nu \leq 1$. Then
\begin{itemize}
\item[(i)] $w_{_{(2,\nu)}}\left(\begin{bmatrix}
A & 0 \\
0 & B
\end{bmatrix}\right) \leq \sqrt{w_{_{(2,\nu)}}^2(A) + w_{_{(2,\nu)}}^2(B)}$.
\item[(ii)] $w_{_{(2,\nu)}}\left(\begin{bmatrix}
0 & A \\
B & 0
\end{bmatrix}\right) \leq  \sqrt{{\|A\|}_{2}^2 + {\|B\|}_{2}^2}$.
\end{itemize}
\end{corollary}
Our final result in this section is a generalization of \cite[Corollary~2]{A.K.2}.
\begin{theorem}\label{T.11.3}
Let $A\in \mathcal{C}_2(\mathcal{H})$ and let $0\leq \nu \leq 1$. Then
\begin{align*}
\frac{1+|1-2\nu|}{\sqrt{2}}w_{_{(2,\nu)}}(A) \leq w_{_{(2,\nu)}}\left(\begin{bmatrix}
0 & \mathfrak{R}_{\nu}(A) \\
\mathfrak{I}_{\nu}(A) & 0
\end{bmatrix}\right) \leq \sqrt{2}w_{_{(2,\nu)}}(A).
\end{align*}
\end{theorem}
\begin{proof}
First, note that
\begin{align}\label{I.1.T.11.3}
\mathfrak{R}_{\nu}(A) + i\mathfrak{I}_{\nu}(A) = 2\nu A \quad \mbox{and} \quad  \mathfrak{R}_{\nu}(A) - i\mathfrak{I}_{\nu}(A) = 2(1-\nu) A^*.
\end{align}
Now, let $U = \frac{1}{\sqrt{2}}\begin{bmatrix}
I & I \\
-I & I
\end{bmatrix}$.
Then, by Proposition \ref{P.9.2}(i), \eqref{I.1.T.11.3},
Corollary \ref{C.10.3}(i), Theorem \ref{T.7.3}(ii), and Theorem \ref{T.7.3}(iv), we have
\begin{align*}
w_{_{(2,\nu)}}\left(\begin{bmatrix}
0 & \mathfrak{R}_{\nu}(A) \\
\mathfrak{I}_{\nu}(A) & 0
\end{bmatrix}\right)
& = 2\nu w_{_{(2,\nu)}}\left(\begin{bmatrix}
0 & \frac{1}{2\nu}\mathfrak{R}_{\nu}(A) \\
\frac{i}{2\nu}\mathfrak{I}_{\nu}(A) & 0
\end{bmatrix}\right)
\\& = 2\nu w_{_{(2,\nu)}}\left(U\begin{bmatrix}
0 & \frac{1}{2\nu}\mathfrak{R}_{\nu}(A) \\
\frac{i}{2\nu}\mathfrak{I}_{\nu}(A) & 0
\end{bmatrix}U^*\right)
\\& = \nu w_{_{(2,\nu)}}\left(\begin{bmatrix}
A & \frac{1-\nu}{\nu}A^* \\
-\frac{1-\nu}{\nu}A^* & -A
\end{bmatrix}\right)
\\& \leq \nu w_{_{(2,\nu)}}\left(\begin{bmatrix}
A & 0 \\
0 & -A
\end{bmatrix}\right)
+ \nu w_{_{(2,\nu)}}\left(\begin{bmatrix}
0 & \frac{1-\nu}{\nu}A^* \\
-\frac{1-\nu}{\nu}A^* & 0
\end{bmatrix}\right)
\\& \leq \sqrt{2}\nu w_{_{(2,\nu)}}(A)
+ \nu w_{_{(2,\nu)}}\left(\begin{bmatrix}
0 & \frac{1-\nu}{\nu}A^* \\
\frac{1-\nu}{\nu}A^* & 0
\end{bmatrix}\right)
\\& = \sqrt{2}\nu w_{_{(2,\nu)}}(A)
+ \sqrt{2}\nu w_{_{(2,\nu)}}\left(\frac{1-\nu}{\nu}A^*\right)
\\& = \sqrt{2}\nu w_{_{(2,\nu)}}(A)
+ \sqrt{2}(1-\nu) w_{_{(2,\nu)}}(A) = \sqrt{2}w_{_{(2,\nu)}}(A).
\end{align*}
Thus
\begin{align}\label{I.2.T.11.3}
w_{_{(2,\nu)}}\left(\begin{bmatrix}
0 & \mathfrak{R}_{\nu}(A) \\
\mathfrak{I}_{\nu}(A) & 0
\end{bmatrix}\right) \leq \sqrt{2}w_{_{(2,\nu)}}(A).
\end{align}
Also, by \eqref{I.1.T.11.3}, Theorem \ref{T.7.3}(iv), and Theorem \ref{T.7.3}(i)-(ii), we have
\begin{align*}
\sqrt{2}\nu w_{_{(2,\nu)}}(A) &= \frac{\sqrt{2}}{2}w_{_{(2,\nu)}}\left(\mathfrak{R}_{\nu}(A) + i\mathfrak{I}_{\nu}(A)\right)
\\& = \frac{1}{2} w_{_{(2,\nu)}}\left(\begin{bmatrix}
0 & \mathfrak{R}_{\nu}(A) + i\mathfrak{I}_{\nu}(A) \\
\mathfrak{R}_{\nu}(A) + i\mathfrak{I}_{\nu}(A) & 0
\end{bmatrix}\right)
\\& = \frac{1}{2} w_{_{(2,\nu)}}\left(\begin{bmatrix}
0 & \mathfrak{R}_{\nu}(A) \\
i\mathfrak{I}_{\nu}(A) & 0
\end{bmatrix}
+ \begin{bmatrix}
0 & i\mathfrak{I}_{\nu}(A) \\
\mathfrak{R}_{\nu}(A) & 0
\end{bmatrix}\right)
\\& \leq \frac{1}{2} w_{_{(2,\nu)}}\left(\begin{bmatrix}
0 & \mathfrak{R}_{\nu}(A) \\
i\mathfrak{I}_{\nu}(A) & 0
\end{bmatrix}\right)
+ \frac{1}{2} w_{_{(2,\nu)}}\left(\begin{bmatrix}
0 & i\mathfrak{I}_{\nu}(A) \\
\mathfrak{R}_{\nu}(A) & 0
\end{bmatrix}\right)
\\& = w_{_{(2,\nu)}}\left(\begin{bmatrix}
0 & \mathfrak{R}_{\nu}(A) \\
\mathfrak{I}_{\nu}(A) & 0
\end{bmatrix}\right),
\end{align*}
and hence
\begin{align}\label{I.3.T.11.3}
\sqrt{2}\nu w_{_{(2,\nu)}}(A) \leq w_{_{(2,\nu)}}\left(\begin{bmatrix}
0 & \mathfrak{R}_{\nu}(A) \\
\mathfrak{I}_{\nu}(A) & 0
\end{bmatrix}\right).
\end{align}
Further, by a similar argument, we have
\begin{align}\label{I.4.T.11.3}
\sqrt{2}(1-\nu) w_{_{(2,\nu)}}(A) \leq w_{_{(2,\nu)}}\left(\begin{bmatrix}
0 & \mathfrak{R}_{\nu}(A) \\
\mathfrak{I}_{\nu}(A) & 0
\end{bmatrix}\right).
\end{align}
So, by \eqref{I.3.T.11.3} and \eqref{I.4.T.11.3}, we obtain
\begin{align}\label{I.4.T.11.3}
\max\left\{\sqrt{2}\nu w_{_{(2,\nu)}}(A), \sqrt{2}(1-\nu) w_{_{(2,\nu)}}(A)\right\} \leq w_{_{(2,\nu)}}\left(\begin{bmatrix}
0 & \mathfrak{R}_{\nu}(A) \\
\mathfrak{I}_{\nu}(A) & 0
\end{bmatrix}\right),
\end{align}
or equivalently,
\begin{align}\label{I.5.T.11.3}
\frac{1+|1-2\nu|}{\sqrt{2}}w_{_{(2,\nu)}}(A) \leq w_{_{(2,\nu)}}\left(\begin{bmatrix}
0 & \mathfrak{R}_{\nu}(A) \\
\mathfrak{I}_{\nu}(A) & 0
\end{bmatrix}\right).
\end{align}
From \eqref{I.2.T.11.3} and \eqref{I.5.T.11.3}, we deduce the desired result.
\end{proof}
\section{An application}
In this section we obtain a refinement of the triangle inequality for the Hilbert--Schmidt norm.
In order to achieve our goal, we need the following lemma. It is well-known and can be found in \cite{Simon}.
\begin{lemma}\label{L.1.4}
Let $X, Y\in \mathcal{C}_2(\mathcal{H})$. Then
\begin{align*}
2\left({\|X\|}_{2}^2 + {\|Y\|}_{2}^2\right) =  {\|X-Y\|}_{2}^2 + {\|X+Y\|}_{2}^2.
\end{align*}
\end{lemma}
\begin{theorem}\label{T.2.4}
Let $A, B \in \mathcal{C}_2(\mathcal{H})$ and let $0\leq \nu \leq 1$. Then
\begin{align*}
{\|A+B\|}_{2} \leq \sqrt{2w_{_{(2,\nu)}}^2\left(\begin{bmatrix}
0 & A \\
B^* & 0
\end{bmatrix}\right) - (1-2\nu)^2{\|A-B\|}_{2}^2}
\leq {\|A\|}_{2} + {\|B\|}_{2}.
\end{align*}
\end{theorem}
\begin{proof}
Applying Lemma \ref{L.1.4} with $X := \frac{A+B}{2}$ and
$Y := \frac{(1-2\nu)(A-B)}{2}$, we get
\begingroup\makeatletter\def\f@size{10}\check@mathfonts
\begin{align*}
2\left({\left\|\frac{A+B}{2}\right\|}_{2}^2 + {\left\|\frac{(1-2\nu)(A-B)}{2}\right\|}_{2}^2\right)
=  {\left\|\nu A+(1-\nu)B\right\|}_{2}^2 + {\left\|(1-\nu)A+\nu B\right\|}_{2}^2.
\end{align*}
\endgroup
This implies
\begingroup\makeatletter\def\f@size{10}\check@mathfonts
\begin{align*}
{\left\|A+B\right\|}_{2}^2 = 2{\left\|\nu A+(1-\nu)B\right\|}_{2}^2 + 2{\left\|(1-\nu)A+\nu B\right\|}_{2}^2
- (1-2\nu)^2{\left\|A-B\right\|}_{2}^2,
\end{align*}
\endgroup
and so
\begingroup\makeatletter\def\f@size{10}\check@mathfonts
\begin{align*}
{\left\|A+B\right\|}_{2}^2 = 2{\left\|\nu A+(1-\nu)B\right\|}_{2}^2 + 2{\left\|(1-\nu)A^*+\nu B^*\right\|}_{2}^2
- (1-2\nu)^2{\left\|A-B\right\|}_{2}^2.
\end{align*}
\endgroup
Therefore, by Lemma \ref{L.6.3}, we obtain
\begingroup\makeatletter\def\f@size{10}\check@mathfonts
\begin{align*}
{\left\|A+B\right\|}_{2}^2 = 2{\left\|\begin{bmatrix}
0 & \nu A+(1-\nu)B \\
(1-\nu)A^*+\nu B^* & 0
\end{bmatrix}\right\|}_{2}^2
- (1-2\nu)^2{\left\|A-B\right\|}_{2}^2,
\end{align*}
\endgroup
or equivalently,
\begingroup\makeatletter\def\f@size{10}\check@mathfonts
\begin{align*}
{\left\|A+B\right\|}_{2}^2 = 2{\left\|\nu\begin{bmatrix}
0 & A \\
B^* & 0
\end{bmatrix} + (1-\nu){\begin{bmatrix}
0 & A \\
B^* & 0
\end{bmatrix}}^*\right\|}_{2}^2
- (1-2\nu)^2{\left\|A-B\right\|}_{2}^2.
\end{align*}
\endgroup
Hence
\begin{align}\label{I.1.T.2.4}
{\left\|A+B\right\|}_{2}^2 = 2{\left\|\mathfrak{R}_{\nu}\left(\begin{bmatrix}
0 & A\\
B^* & 0
\end{bmatrix}\right)\right\|}_{2}^2
- (1-2\nu)^2{\left\|A-B\right\|}_{2}^2.
\end{align}
Since ${\left\|\mathfrak{R}_{\nu}\left(\begin{bmatrix}
0 & A\\
B^* & 0
\end{bmatrix}\right)\right\|}_{2} \leq w_{_{(2,v)}}\left(\begin{bmatrix}
0 & A \\
B^* & 0
\end{bmatrix}\right)$, by \eqref{I.1.T.2.4} it follows that
\begin{align}\label{I.2.T.2.4}
{\left\|A+B\right\|}_{2}^2 \leq 2w_{_{(2,v)}}^2\left(\begin{bmatrix}
0 & A \\
B^* & 0
\end{bmatrix}\right)- (1-2\nu)^2{\left\|A-B\right\|}_{2}^2.
\end{align}
Also, by Lemma \ref{L.6.3} and the triangle inequality for the norm ${\|\cdot\|}_{2}$, we have
\begingroup\makeatletter\def\f@size{10}\check@mathfonts
\begin{align*}
2&w_{_{(2,v)}}^2\left(\begin{bmatrix}
0 & A \\
B^* & 0
\end{bmatrix}\right)- (1-2\nu)^2{\left\|A-B\right\|}_{2}^2
\\&= 2\displaystyle{\sup_{\theta \in \mathbb{R}}}
{\left\|\mathfrak{R}_{\nu}\left(e^{i\theta}\begin{bmatrix}
0 & A\\
B^* & 0
\end{bmatrix}\right)\right\|}_{2}^2
- (1-2\nu)^2{\left\|A-B\right\|}_{2}^2
\\& = 2\displaystyle{\sup_{\theta \in \mathbb{R}}}
{\left\|\begin{bmatrix}
0 & \nu e^{i\theta} A+(1-\nu)e^{-i\theta}B \\
(1-\nu)e^{-i\theta}A^*+\nu e^{i\theta}B^* & 0
\end{bmatrix}\right\|}_{2}^2
- (1-2\nu)^2{\left\|A-B\right\|}_{2}^2
\\& = 2\displaystyle{\sup_{\theta \in \mathbb{R}}}\left(
{\left\|\nu e^{i\theta} A+(1-\nu)e^{-i\theta}B\right\|}_{2}^2
+ {\left\|(1-\nu)e^{-i\theta}A^*+\nu e^{i\theta}B^*\right\|}_{2}^2\right)
- (1-2\nu)^2{\left\|A-B\right\|}_{2}^2
\\& \leq 2\Big(\nu{\left\|A\right\|}_{2} + (1-\nu){\left\|B\right\|}_{2}\Big)^2
+ 2\Big((1-\nu){\left\|A\right\|}_{2} + \nu{\left\|B\right\|}_{2}\Big)^2
- (1-2\nu)^2\Big|{\left\|A\right\|}_{2}-{\left\|B\right\|}_{2}\Big|^2
\\& = {\|A\|}_{2}^2 + 2 {\|A\|}_{2}{\|B\|}_{2} +{\|B\|}_{2}^2,
\end{align*}
\endgroup
and hence
\begin{align}\label{I.3.T.2.4}
2&w_{_{(2,v)}}^2\left(\begin{bmatrix}
0 & A \\
B^* & 0
\end{bmatrix}\right)- (1-2\nu)^2{\left\|A-B\right\|}_{2}^2
\leq \left({\|A\|}_{2} + {\|B\|}_{2}\right)^2.
\end{align}
Now, by \eqref{I.2.T.2.4} and \eqref{I.3.T.2.4}, we deduce the desired result.
\end{proof}
\begin{remark}\label{R.3.4}
Theorem \ref{T.1.3} is a generalization of a result due to Aldalabih and Kittaneh \cite{A.K.2}.
In fact if $\nu = 1/2$, then
\begin{align*}
{\|A+B\|}_{2} \leq \sqrt{2}w_{_{2}}\left(\begin{bmatrix}
0 & A \\
B^* & 0
\end{bmatrix}\right) \leq {\|A\|}_{2} + {\|B\|}_{2},
\end{align*}
which has been proven in \cite[Theorem~7]{A.K.2}.
For the usual operator norm and the Schatten $p$-norm,
related results have been given in \cite[Theorem~2.3]{K.M.Y} and \cite[Theorem~4]{A.A}, respectively.
\end{remark}
\bibliographystyle{amsplain}

\end{document}